\documentclass{amsart}
\usepackage[english]{babel}
\usepackage[latin1]{inputenc}

\usepackage{mathpple}

\usepackage{newlfont}
\usepackage{amsmath}
\usepackage{graphicx}
\usepackage{amssymb}
\usepackage{amsmath}
\usepackage{latexsym}
\usepackage{amsthm}
\usepackage{stmaryrd}
\usepackage{braket}
\usepackage{enumerate}
\usepackage{booktabs}
\usepackage{mathtools}
\usepackage{array}
\usepackage{paralist}
\usepackage{color}
\usepackage{cleveref}

\usepackage[a4paper,top=3cm,bottom=3cm,left=3.5cm,right=3.5cm,%
bindingoffset=5mm]{geometry}

\usepackage{pxfonts}

\newtheorem{Theorem}{Theorem}[section]
\newtheorem{Lemma}[Theorem]{Lemma}
\newtheorem{Cor}[Theorem]{Corollary}
\newtheorem{Prop}[Theorem]{Proposition}
\theoremstyle{definition}
\newtheorem{Def}[Theorem]{Definition}
\newtheorem{Rem}[Theorem]{Remark}

\newtheorem{Ex}[Theorem]{Example}

\DeclareMathOperator{\Syz}{Syz}
\DeclareMathOperator{\Sym}{Sym}
\DeclareMathOperator{\Spec}{Spec}
\DeclareMathOperator{\depth}{depth}
\DeclareMathOperator{\rank}{rank}
\DeclareMathOperator{\Hom}{Hom}
\DeclareMathOperator{\freerank}{frk}

\DeclareMathOperator{\chara}{char}

\newcommand{\kk}{k}            
\newcommand{\mmodgr}[1]{\mathrm{mod}_{\mathbb{Z}}(#1)}   
\newcommand{\proj}[1]{\mathrm{proj}_{\mathbb{Z}}(#1)}   
\newcommand{\Gl}[1]{\mathrm{GL}(#1)}   
\newcommand{\Sl}[1]{\mathrm{SL}(#1)}   
\newcommand{\dd}{\mathrm{d}}  

\newcommand{\fundspace}{F}

\title{Differential symmetric signature in high dimension}

\author{Holger Brenner} 
\address{{\small Holger Brenner, Institut f\"ur Mathematik, Universit\"at Osnabr\"uck, Albrechtstrasse 28a, 49076 Osnabr\"uck, Germany}}
\email{{\small holger.brenner@uni-osnabrueck.de}}

\author{Alessio Caminata}
\thanks{The second author is supported by European Union's Horizon 2020 research and innovation programme under grant agreement No 701807.}
\address{{\small Alessio Caminata, Institut de Matem\`{a}tica, Universitat de Barcelona\\ Gran Via de les Corts Catalanes 585, 08007 Barcelona, Spain}}
\email{{\small caminata@ub.edu}}

\begin{document}


\begin{abstract}
We study the differential symmetric signature, an invariant of rings of finite type over a field, introduced in a previous work by the authors in an attempt to find a characteristic-free analogue of the F-signature.
We compute the differential symmetric signature for invariant rings $\kk[x_1,\dots,x_n]^G$ where $G$ is a finite small subgroup of $\Gl{n,\kk}$ and for hypersurface rings $\kk[x_1,\dots,x_n]/(f)$ of dimension $\geq3$ with an isolated singularity.
In the first case, we obtain the value $1/|G|$, which coincides with the F-signature and generalizes a previous result of the authors for the two-dimensional case. In the second case, following an argument by Bruns, we obtain the value $0$, providing an example of a ring where differential symmetric signature and F-signature are different.
\end{abstract}

\maketitle

\section*{Introduction}

\par In the seminal paper \cite{Kun69}, Kunz introduces the idea of studying singularities in positive characteristic by looking at numerical functions defined using the Frobenius homomorphism.
Given a Noetherian local ring $(R,\mathfrak{m},\kk)$ containing a field of positive characteristic, he studies the function $q\mapsto\ell_R(R/\mathfrak{m}^{[q]})$, which is now called \emph{Hilbert-Kunz function} and uses it to characterize the regularity of the ring $R$. Here $q$ is a power of the characteristic and $\mathfrak{m}^{[q]}$ is the Frobenius power of the maximal ideal.

\par Smith and Van den Bergh \cite{SVB96}, and Seibert \cite{Sei97} take a slight shift and look instead at the asymptotic splitting behaviour of the Frobenius endomorphism. 
More precisely, we consider  $R$ as $R$-module with the multiplicative structure given by the $e$-th iteration of the Frobenius homomorphism.
We denote this module by $^e\!R$, and look at decompositions $^e\!R\cong R^{a_e}\oplus M_e$, where $M_e$ does not contain free $R$-summands.
The integer $a_e$ is called \emph{free rank} of $^e\!R$ and denoted by $\freerank_R{^e\!R}$.
Smith and Van den Bergh, and Seibert study the asymptotic behaviour  of the sequence
\begin{equation}\label{eq-Fsignaturesequence}
\left\{
\frac{\freerank_R{^e\!R}}{\rank_R{^e\!R}}\right\}_{e\in\mathbb{N}}
\end{equation}
where $\rank_R{^e\!R}$ is the usual rank of $R$-modules, and prove that it converges if $R$ is of finite Cohen-Macaulay type (Seibert) or of finite F-representation type and strongly F-regular (Smith and Van den Bergh).

\par Huneke and Leuschke \cite{HL02} continue this approach  and study the previous sequence also for larger classes of rings.
They define the \emph{F-signature} of  $R$ as the limit of the sequence \eqref{eq-Fsignaturesequence} for $e\rightarrow+\infty$ and prove the existence of the limit in several cases.
We mention that in the same years, Watanabe and Yoshida \cite{WY04} define another numerical invariant,  the \emph{minimal relative Hilbert-Kunz multiplicity}, that turns out to be equivalent to the F-signature \cite{Yao06}.

\par The F-signature has drawn the attention of several researchers, and actually it proves to be an important numerical invariant of the ring which provides delicate information concerning its singularities.
To mention a few results, we recall that the existence of the limit is known in greater generality \cite{Tuc12}, then one has that $s(R)=1$ if and only if the ring is regular \cite{HL02, WY00}, and $s(R)>0$ if and only if $R$ is strongly F-regular \cite{AL03}.

\par Given the importance of F-signature, it seems a natural step to study the ratio $\freerank_R/\rank_R$ also for other families of $R$-modules, not necessarily coming from the Frobenius homomorphism.
With this motivation in mind, in a previous work  \cite{BC17} the authors introduce the differential symmetric signature.
\par Let $R$ be a normal domain of finite type over a field $\kk$ (of any characteristic).
For any natural number $q$ we consider $q$-th reflexive symmetric powers
\begin{equation*}
\mathcal{C}^q=\left(\Sym_R^q(\Omega_{R/\kk})\right)^{**}
\end{equation*}
of the module of K\"ahler differentials $\Omega_{R/\kk}$ of the ring $R$ over $\kk$ and study the asymptotic behaviour of the ratio $\freerank_R/\rank_R$ for this class of modules, rather than for $^e\!R$ as it is for the $F$-signature.
If the limit
\[ 
\lim_{N\rightarrow+\infty} \frac{\sum_{q=0}^N\freerank_R\mathcal{C}^q}{\sum_{q=0}^N\rank_R\mathcal{C}^q}
\]
exists, it is called the \emph{differential symmetric signature} of $R$ and denoted by $s_{\dd\sigma}(R)$.

\par In the paper \cite{BC17}, the authors compute the differential symmetric signature of two classes of rings: two-dimensional Kleinian singularities and cones over elliptic curves.
In both cases, the values obtained coincide with the F-signature of such rings in positive characteristic.
However, the results and the methods  there are limited to the two-dimensional situation.
In this paper, we compute the differential symmetric signature of two large classes of $d$-dimensional rings for $d\geq3$: quotient singularities and hypersurfaces with an isolated singularity.

\par A quotient singularity is an invariant ring of the form $\kk[x_1,\dots,x_n]^G$, where $G\subseteq \Gl{n,\kk}$ is a finite small group, i.e. it contains no pseudo-reflections, whose order is coprime to the characteristic of the base field $\kk$.
Quotient singularities are normal and Cohen-Macaulay.
Watanabe and Yoshida \cite{WY04} proved that in positive characteristic the F-signature of such rings is equal to $\frac{1}{|G|}$, provided that $\chara\kk\nmid |G|$, where $|G|$ is the order of the group.
\par Kleinian singularities are special two-dimensional singularities, that is $G\subseteq\Sl{2,\kk}$.
As mentioned before, the authors \cite{BC17}  proved that the differential symmetric signature of Kleinian singularities is also $\frac{1}{|G|}$,  if  $\chara\kk\nmid|G|$.

\par In this paper, we generalize the latter result to higher dimension and to the whole class of quotient singularities.
Namely, we prove in Theorem \ref{theorem:signatureofquotientsingularities} that 
\begin{equation*}
s_{\dd\sigma}\left(\kk[x_1,\dots,x_n]^G\right)=\frac{1}{|G|}
\end{equation*}
for any $n\geq2$ and any finite small group $G\subseteq\Gl{n,\kk}$ such that $\chara\kk\nmid|G|$.
Notice that we allow the characteristic of $\kk$ to be $0$.
Thus, we obtain an analogue for the differential symmetric signature to the result of Watanabe and Yoshida.

\par The second main result of this paper is the computation of the differential symmetric signature for domains of the form $R=\kk[x_1,\dots,x_n]/(f)$, where $n\geq4$ and $f$ is a non-zero homogeneous polynomial whose degree is at least $2$ and is not a multiple of $\chara\kk$.
The ring $R$ is the coordinate ring of a hypersurface  in $n$-dimensional affine space over $\kk$, and we assume that it has an isolated singularity.
Using ideas suggested by Bruns, we prove in Theorem \ref{theorem:freerankiszero} that the differential symmetric signature of $R$ is $0$.
Then, we use this result to provide an example of a ring where the differential symmetric signature and the F-signature are different (Example \ref{ex:differentsignatures}).

\par The structure of the paper is the following.
In Section \ref{section-preliminaries}, we review the definition of differential symmetric signature together with some known results.
In Section \ref{section-invariants}, we compute the differential symmetric signature of quotient singularities (Theorem \ref{theorem:signatureofquotientsingularities}) .
In Section \ref{section-hypersurfaces}, we concentrate on normal hypersurfaces with an isolated singularity and we prove that their differential symmetric signature is zero (Theorem \ref{theorem:freerankiszero}).

\section{Preliminaries}\label{section-preliminaries}

\par We recall the definition of differential symmetric signature from \cite{BC17}.
Since in this paper we are mainly interested in graded rings, we restrict to this setting from the beginning, although the differential symmetric signature can be defined for local rings in a similar way.
\par Let $\kk$ be an algebraically closed field, and let $R$ be a $\mathbb{Z}$-graded Noetherian $\kk$-domain of dimension $d$.
We will work in the category $\mmodgr{R}$ of graded $R$-modules. 
Given a finitely generated graded module $M$ we denote by $\rank_R M$ the \emph{rank} of $M$, and by $\freerank_R M$ the \emph{free rank} of $M$, that is

\begin{equation}\label{eq:freerank}\begin{split}
\freerank_R M:=\max\{n\in\mathbb{N}: \ \exists &\text{ a homogeneous of degree }0\text{ surjection }\varphi:M\twoheadrightarrow F, \\ &\text{ with } F \text{ free graded }R\text{-module of rank }n  \}.
\end{split}\end{equation}
Since the category $\mmodgr{R}$ is Krull-Schmidt \cite{Ati56}, if $M$ has finite free rank, then we can write it as $M\cong R^{\freerank_RM}\oplus N$, where the module $N$ has no free direct $R$-summands.
In particular, it follows that $\freerank_RM\leq\rank_RM$.

\par We denote by $\Omega_{R/\kk}$ the \emph{module of K\"ahler differentials} of $R$ over $\kk$, and by $(-)^{*}$ the functor $\Hom_R(-,R)$.
The reflexive hull $\Omega_{R/\kk}^{**}$ is called \emph{module of Zariski differentials} of $R$ over $\kk$.
We recall that under some mild conditions we have that $\Omega_{R/\kk}$ is free if and only if $R$ is regular.
For every natural number $q$, we define the graded $R$-module
\begin{equation*}
\mathcal{C}^q:=\left(\Sym^q_R(\Omega_{R/\kk})\right)^{**}.
\end{equation*}
 
\begin{Def}
	The real number
	\begin{equation*}
	 s_{\dd\sigma}(R):=\lim_{N\rightarrow+\infty}\frac{\sum_{q=0}^N\freerank_R\mathcal{C}^q}{\sum_{q=0}^N\rank_R\mathcal{C}^q}
	\end{equation*}
	is called \emph{differential symmetric signature} of $R$, provided the limit exists.
\end{Def} 

\begin{Theorem}[\cite{BC17}]
The following facts hold.
\begin{enumerate}
	\item If $R$ is regular, then  $s_{\dd\sigma}(R)=1$.
	\item If $R=\kk[x,y]^G$, where $G\subseteq\Sl{2,\kk}$ is a finite small group such that $\chara\kk\nmid|G|$ then $ s_{\dd\sigma}(R)=\frac{1}{|G|}$.
	\item If $R=\kk[x,y,z]/(f)$ is the coordinate ring of a plane elliptic curve and $\chara\kk\neq2,3$, then $ s_{\dd\sigma}(R)=0$.
\end{enumerate}
\end{Theorem}

\begin{Rem}	We mention that there is an alternative version of the symmetric signature defined using the syzygy module $\Syz_R^d(\kk)$ instead of $\Omega_{R/\kk}$, and called \emph{syzygy symmetric signature}.
	Results \textit{(1)} and \textit{(2)} of the previous theorem are true also for this variant of the symmetric signature. 
	Moreover the second author and Katth\"an \cite{CK17} proved that the syzygy symmetric signature of two-dimensional cyclic quotient singularities is the reciprocal of the order of the acting group.
		However, we will not consider the syzygy symmetric signature in this paper.
\end{Rem}

\par We conclude this section with the following general lemma which gives a geometric interpretation of the reflexive hull.

\begin{Lemma}\label{lemma:reflexivehull}
Let $R$ be a normal Noetherian domain of dimension $d\geq2$ over an algebraically closed field $\kk$ and we denote by $X^{(1)}$ the subset of $X=\Spec R$ consisting of all prime ideals of height one.
Let $M$ be a torsion-free finitely generated $R$-module, then the reflexive hull of $M$ is
\begin{equation*}
M^{**}=\lim_{X^{(1)}\subseteq U}\Gamma(U,\widetilde{M}),
\end{equation*}
where the limit runs over all open subsets $U$ containing all prime ideals of height one. 
\end{Lemma}

\begin{proof}
Let $U$ be a Zariski-open set of $X$ containing $X^{(1)}$.
Let $I$ be an ideal of $R$ such that $U=X\setminus\mathcal{V}(I)$, where $\mathcal{V}(I):=\{\mathfrak{p}\in X: \ \mathfrak{p}\supseteq I\}$.
Since $U$ contains $X^{(1)}$, every prime ideal of $\mathcal{V}(I)$ has height at least $2$.
\par First, assume that $M$ is reflexive. 
Then, $M$ satisfies $(S_2)$.
It follows that $\depth M_{\mathfrak{p}}\geq2$ for every prime ideal $\mathfrak{p}\in\mathcal{V}(I)$.
Therefore, we have that 
\begin{equation*}
\mathrm{grade}(I,M)=\inf\{\depth M_{\mathfrak{p}}: \ \mathfrak{p}\in\mathcal{V}(I)\}=2,
\end{equation*}
and consequently $H_I^0(M)=H_I^1(M)=0$.
Then, the short exact sequence
 \begin{equation*}
  0\rightarrow H^0_{I}(M)\rightarrow M\rightarrow \Gamma(U,\widetilde{M})\rightarrow H^{1}_{I}(M)\rightarrow 0,
 \end{equation*}
yields the isomorphism $M\cong \Gamma(U,\widetilde{M})$.
\par Now, let $M$ be torsion-free and consider the short exact sequence
\begin{equation*}
0\rightarrow M \rightarrow M^{**} \rightarrow T \rightarrow 0.
\end{equation*}
For any prime ideal $\mathfrak{p}$ of height $0$ or $1$, the local ring $R_{\mathfrak{p}}$ is regular of dimension $\leq1$, since $R$ is normal, hence regular in codimension $1$.
Therefore, the torsion-free module $M_{\mathfrak{p}}$ is reflexive, and $T_{\mathfrak{p}}=0$.
Since the limit $\lim_{X^{(1)}\subseteq U}\Gamma(U,\widetilde{M})$ is taken over all open subsets which contain all prime ideals of height $0$ and $1$, the claim is proved.
\end{proof}

\section{Symmetric signature of invariant rings}\label{section-invariants}
\par Let $\kk$ be an algebraically closed field and let $ \fundspace=\kk^n$ be a $\kk$-vector space of dimension $n\geq2$.
We recall that an element $\sigma\in \Gl{F}$ of finite order is called a \emph{pseudo-reflection} provided the fixed subspace $\{v\in F:\ \sigma(v)=v\}$
has codimension one in $F$.
We say a subgroup of $\Gl{F}$ is \emph{small} if it contains no pseudo-reflections.
Let $G\subseteq\Gl{F}$ be a finite small group whose order is coprime with the characteristic of $\kk$.
The natural action of $G$ on $F$ extends naturally to the polynomial ring
$S=\kk[ \fundspace]=\kk[x_1,\dots,x_n]$.
More precisely, we fix a basis $\{e_1,\dots,e_n\}$ of $ \fundspace $, then we identify $e_i$ with the variable $x_i$ of $S$ and we define $\sigma(x_i)=\sigma(e_i)$ for every $\sigma\in G$.
We denote by $R=S^G$ the invariant ring of $S$ under this action.
$R$ is called \emph{quotient singularity}, and is a Noetherian Cohen-Macaulay graded normal domain of dimension $n$ (see e.g. \cite[Theorem 4.1]{BD08}).

\par Watanabe and Yoshida \cite{WY04} proved that in positive characteristic the F-signature of $R$ is equal to $\frac{1}{|G|}$, and the authors \cite{BC17} proved that the differential symmetric signature of $R$ is also $\frac{1}{|G|}$ if $n=2$ and $G\subseteq\Sl{F}$.
The goal of this section is to prove that $s_{\dd \sigma}(R)=\frac{1}{|G|}$ for any $n\geq2$ and any finite small subgroup $G\subseteq\Gl{ \fundspace }$ with $\chara\kk\nmid |G|$.

\subsection{Auslander correspondence}
\par We will need the Auslander correspondence between irreducible $\kk$-representations of $G$ and indecomposable $R$-direct summands of $S$.
We reviewed this correspondence in details in the paper \cite{BC17}, here we will limit ourselves to fix the notation and mention the results we need. 
In addition, the reader interested in the proof of the Auslander correspondence in the graded setting may look at the paper of Iyama and Takahashi \cite{IT13}.


\par We consider the natural functor $\mathcal{F}$ from the category of $\kk$-representations of $G$ to the category $\proj{S*G}$ of finitely generated graded projective $S*G$-modules, 
where $S*G$ denotes the skew-group ring of $S$ and $G$.
For every $\kk$-representation $V$, we have $\mathcal{F}(V):=S\otimes_{\kk}V$.
The functor $\mathcal{F}$ has a right adjoint given by $\mathcal{F}'(P):=P\otimes_{S}\kk$.
The functors $\mathcal{F}$ and $\mathcal{F}'$ are an adjoint pair \cite[Lemma 10.1]{Yos90}, but in general they are not an equivalence of categories.

\par We define the functors $\mathcal{G}:\proj{S*G}\rightarrow\mathrm{Add}_R(S)$, $\mathcal{G}(M)=M^G$, and $\mathcal{G}':\mathrm{Add}_R(S)\rightarrow\proj{S*G}$, $\mathcal{G}'(N)=(S\otimes_RN)^{**}$, where $(-)^{*}:=\Hom_S(-,S)$, and $\mathrm{Add}_R(S)$ is the category whose objects are graded $R$-direct summands of $S$ and their direct sums.
Notice that the objects of $\mathrm{Add}_R(S)$ are $R$-reflexive.
The functors $\mathcal{G}$ and $\mathcal{G}'$ give an equivalence of categories \cite[Corollary 4.4]{IT13}
\begin{equation*}
\proj{S*G}\cong\mathrm{Add}_R(S).
\end{equation*}

\par The composition of the functors $\mathcal{F}$ and $\mathcal{G}$ is the \emph{Auslander functor} $\mathcal{A}(V):=(S\otimes_{\kk}V)^G$.
Its right adjoint is the functor
$\mathcal{A}'(N):=(S\otimes_RN)^{**}\otimes_{S}\kk$, where $N$ is an object in $\mathrm{Add}_R(S)$.
\begin{Theorem}[Auslander correspondence]\label{theorem:equivalenceofcategories}
	The functors $\mathcal{A}$ and $\mathcal{A}'$ have the following properties.
	\begin{compactenum}[1)]
		\item $\mathcal{A}(V)$ is an indecomposable $R$-module if and only if $V$ is an irreducible representation.
		\item $\mathcal{A}(\mathcal{A}'(M)) \cong M$ for every graded finitely generated $R$-module $M$ in $\mathrm{Add}_R(S)$.
		\item $\mathcal{A}'(\mathcal{A}(V)) \cong V$ for every $\kk$-representation $V$.
		\item $\rank_R\mathcal{A}(V)=\dim_{\kk}V$ for every $\kk$-representation $V$.
		\item If $V$ is a $\kk$-representation, and $a_0$ denotes the multiplicity of the trivial representation in $V$, then $a_0=\freerank_R\mathcal{A}(V)$.
	\end{compactenum}
\end{Theorem} 

\begin{Theorem}[\cite{BC17}, Theorem 2.12]\label{theoremsymcommuteswithauslander}
	The Auslander functor commutes with reflexive symmetric powers, that is for every $\kk$-representation $V$ of $G$, it holds 
	\begin{equation*}
	\Sym_R^{q}(\mathcal{A}(V))^{**}\cong\mathcal{A}(\Sym^{q}_{\kk}(V)). 
	\end{equation*}
\end{Theorem}

\subsection{Invariants of K\"ahler differentials}
\par Let $\Omega_{S/\kk}$ be the module of K\"ahler differentials of $S$ over $\kk$.
Since $S$ is a polynomial ring, $\Omega_{S/\kk}$ is $S$-free and generated by the elements $\dd x_1,\dots, \dd x_n$.
The module $\Omega_{S/\kk}$ comes with a natural $G$-action, given by 

\begin{equation*}
f\dd h\mapsto \sigma(f)\dd \sigma(h),
\end{equation*}
for every $f,h\in S$ and $\sigma\in G$.
More explicitely, we have
\begin{equation*}
\sum_{i}f_i \dd x_i \mapsto \sum_i \sigma(f_i) \dd \sigma(x_i).
\end{equation*} 
With this action, $\Omega_{S/\kk}$ has a structure of  $S*G$-module, and we can consider its invariant submodule $\Omega_{S/\kk}^G$.
We will prove in Proposition \ref{prop:invariantdifferentials} that this is the module of Zariski differentials $\Omega_{R/\kk}^{**}$, where $(-)^{*}:=\Hom_R(-,R)$.

\begin{Rem}\label{rem:differentialisfundamental}
	The representation corresponding to the action of $G$ on $\Omega_{S/\kk}$ is the fundamental representation of $G$ on $ \fundspace $ which defines the invariant ring. The corresponding action on	
	 $\mathcal{F} (\fundspace) \cong S^n $ is just given by sending $e_i$ to $\sigma(e_i)$. But the natural action of $\sigma$ on the free $S$-module $\Omega_{S/\kk}$ with basis given by $\dd x_1,\dots,\dd x_n$ is given by $\dd x_i\mapsto \sigma(\dd x_i)=\dd \sigma(x_i)$, so it is isomorphic to the fundamental representation.
\end{Rem}

\par Let $\mathbb{A}^n_{\kk}=\Spec S$, and $X=\Spec R=\mathbb{A}^n_{\kk}/G$ be the quotient space.
Since $G$ is small, its action has no fix-points in codimension one.
In particular there exists a $G$-invariant open subset $U'\subseteq\mathbb{A}^n_{\kk}$ containing all points of codimension one and having no fixed points. Let $U\subseteq X$ be the corresponding quotient. 
Since $G$ is small, the projection map $\pi:U'\rightarrow U$ is unramified, thanks to the following result (cf. \cite[Theorem B.29]{LW12}).

\begin{Theorem}\label{theorem:smallisunramified}
Let $G \subseteq \Gl{\fundspace }$ be a finite group of linear automorphisms of
a finite-dimensional vector space $\fundspace $ over a field $\kk$. Set $S=\kk[\fundspace ]$ and $R= S^G$. Assume that $|G|$ is invertible in $\kk$. Then $R\hookrightarrow S$ is unramified in codimension one if and only if $G$ is small.
\end{Theorem}

\begin{Rem}\label{rem:specialopenset}
	The open set $U$ plays a special role, in fact by Lemma \ref{lemma:reflexivehull} we have that evaluation over $U$ corresponds to taking reflexive hulls.
	More precisely, if $M$ is a torsion-free $R$-module which is locally free on $U$, then
	\begin{equation*}
M^{**}=\Gamma(U,\widetilde{M}).		
	\end{equation*}
\end{Rem}

\begin{Rem}\label{rem:Gisprojection}
The functor $\mathcal{G}'$ is the algebraic counterpart of the projection map $\pi:U'\rightarrow U$ on the fix-point-free open sets.
In fact, if $N$ is a finitely generated $R$-module, and $\widetilde{N}$ is the corresponding sheaf on $U$, then the pullback $\pi^{*}\widetilde{N}$ on $U'$ is the sheaf corresponding to the $S$-module  $\mathcal{G}'(N)=(S\otimes_RN)^{**}$.
\end{Rem}

\par The following proposition can be seen as a generalization of a result of Martsinkovsky \cite{Mar90}, who proved it for two-dimensional quotient singularities.

\begin{Prop}\label{prop:invariantdifferentials}
	Let $\Omega_{R/\kk}^{**}$ be the module of Zariski differentials of $R$ over $\kk$. 
	Then we have  $\Omega_{S/\kk}^G\cong \Omega_{R/\kk}^{**}$.
\end{Prop}

\begin{proof}
	There is a natural $R$-module homomorphism $\Omega_{R/\kk}\rightarrow\Omega_{S/\kk}^G$. Thanks to Lemma \ref{lemma:reflexivehull} and to Remark \ref{rem:specialopenset}, it is enough to check that it is an isomorphism over $U$, i.e. in codimension one.
	\par Let $\mathfrak{q}\in U'$ be a prime ideal of height $0$ or $1$, and let $\mathfrak{p}:=\mathfrak{q}\cap R\in U$ be its restriction to $R$. By localization properties of the module of differentials, we have $(\Omega_{S/\kk})_{\mathfrak{q}}\cong \Omega_{S_{\mathfrak{q}}/\kk}$ and $(\Omega_{R/\kk})_{\mathfrak{p}}\cong \Omega_{R_{\mathfrak{p}}/\kk}$.
	Since $G$ is small, by Theorem \ref{theorem:smallisunramified} we have that $R_{\mathfrak{p}}\hookrightarrow S_{\mathfrak{q}}$ is unramified, that is $\Omega_{S_{\mathfrak{q}}/R_{\mathfrak{p}}}=0$.
	For ease of notation we set $R=R_{\mathfrak{p}}$ and $S=S_{\mathfrak{q}}$.
	The relative cotangent sequence yields
	\begin{equation}\label{eq:conormalsequence}
	S\otimes_R\Omega_{R/\kk}\xrightarrow{\varphi}\Omega_{S/\kk}\rightarrow0,
	\end{equation}
	since $\Omega_{S/R}=0$.
	The rings $R$ and $S$ are regular, so $\Omega_{S/\kk}$ and $\Omega_{R/\kk}$ are free modules of rank $n$ over $S$ and $R$ respectively.
	Moreover $\Omega_{S/\kk}$ and $S\otimes_R\Omega_{R/\kk}$ are $S*G$-modules, and the map $\varphi$ is an $S*G$-module homomorphism. In fact, for every $\sigma\in G$, $s\in S$, and $r\in R$ we have
	\begin{equation*}
	\varphi(\sigma(s\otimes\dd r))=\varphi(\sigma(s)\otimes\dd\sigma(r))=\sigma(s)\dd r=\sigma(s\dd r)=\sigma(\varphi(s\otimes\dd r)).
	\end{equation*}
		The $S*G$-module $S\otimes_R\Omega_{R/\kk}$ is free of rank $n$.
	In particular, it is also reflexive, and we can rewrite \eqref{eq:conormalsequence} as
		\begin{equation*}
		(S\otimes_R\Omega_{R/\kk})^{**}\xrightarrow{\overline{\varphi}}\Omega_{S/\kk}\rightarrow0,
		\end{equation*}
	where $\overline{\varphi}$ is the composition of $\varphi$ with the inverse of the canonical isomorphism $S\otimes_R\Omega_{R/\kk}\rightarrow (S\otimes_R\Omega_{R/\kk})^{**}$. 	
	The homomorphism $\overline{\varphi}$ is a surjective map  between two free $S$-modules of the same rank $n$, hence it is also injective.
	Therefore $\overline{\varphi}$ is an isomorphism of $S*G$-modules.
	Since $\mathcal{G}$ and $\mathcal{G}'$ are an equivalence of categories, the isomorphism $(S\otimes_R\Omega_{R/\kk})^{**}\cong\Omega_{S/\kk}$ implies
	$\Omega_{S/\kk}^G\cong\Omega_{R/\kk}^{**}$ as required.
\end{proof}

\par In the complete local case in characteristic $0$, an analogous statement has been proved by Platte \cite{Pla80}.

\subsection{Differential symmetric signature of quotient singularities}

\begin{Theorem}\label{theorem:signatureofquotientsingularities}
	Let $F$ be a $\kk$-vector space of dimension $n\geq2$, and let $G\subseteq\Gl{ \fundspace }$ be a finite small group such that $\chara\kk\nmid|G|$. 
	Let $S=\kk[\fundspace ]$, and let $R=S^G$ be the invariant ring.
	Then the differential symmetric signature of $R$ is
	\begin{equation*}
 s_{\dd\sigma}(R)=\frac{1}{|G|}.
	\end{equation*}
\end{Theorem}

\begin{proof}
	Let $M:=\Omega_{R/\kk}^{**}$ be the module of Zariski differentials of $R$ over $\kk$. For every natural number $q$, we fix the notation for the numbers involved in the limit defining the differential symmetric signature:
	\begin{equation*}
	a_q:=\freerank_R\Sym_R^q(M)^{**}\ \text{ and } \  
	b_q:=\rank_R\Sym_R^q(M)^{**}.
		\end{equation*}
	By Proposition \ref{prop:invariantdifferentials} and Remark \ref{rem:differentialisfundamental} we have that $M$ corresponds to the fundamental representation $\fundspace$ via Auslander correspondence.
	Therefore thanks to Theorem \ref{theoremsymcommuteswithauslander} we can interpret  the numbers $a_q$ and $b_q$ as 
	\begin{equation*}
	a_q=\dim_{\kk}(\Sym_{\kk}^q(\fundspace))^G=\dim_{\kk}(R_q) \ \text{ and } \
	b_q=\dim_{\kk}(\Sym_{\kk}^q(\fundspace))=\dim_{\kk}(S_q),
		\end{equation*}
	where $R_q$ (resp. $S_q$) denotes the $q$-th degree part of $R$ (resp. $S$).
	In other words, $a_q$ and $b_q$ are the values of the Hilbert functions of $R$ and $S$ respectively.
	In particular, the limit defining the differential symmetric signature involves the cumulative Hilbert functions $\sum_{q=0}^Na_q$ and $\sum_{q=0}^N b_q$ of $R$ and $S$.
	\par The cumulative Hilbert functions of $R$ and $S$ are given by quasi-polynomials of the same degree $n$ and constant leading terms $c$ and $c'$ (see e.g. \cite[Theorem 2.7 (2)]{DS90}), that is
	\begin{equation}\label{eq:iteratedHfunctions}
\sum_{q=0}^N a_q = cN^n+O(N^{n-1}) \ \text{ and } \
\sum_{q=0}^N b_q = c'N^n+O(N^{n-1}).
	\end{equation}
	It follows that the differential symmetric signature exists, and it is equal to $c/c'$.
	\par Since $Q(R) = Q(S)^G$ and $Q(R) \subseteq Q(S)$ is a Galois extension of degree $m=|G|$, there exist homogeneous elements $s_1,\dots, s_m$  of $S$ of degrees $h_1,\dots,h_m$ such that 
	\begin{equation*}
	\bigoplus_{i=1}^m Rs_i \subseteq S 
		\end{equation*}
		and the quotient is torsion.
	Then, we can write the $q$-th value $b_q$ of the Hilbert function of $S$ as
		$b_q= \sum_{i=1}^m a_{q-h_i}$ up to the torsion part.
	Summing up for $q=0$ to $N$, we obtain an equality for the cumulative Hilbert functions
	\begin{equation*}
	\sum_{q=0}^N b_q= \sum_{q=0}^N \sum_{i=1}^m a_{q-h_i} + O(N^{n-1}) .
	\end{equation*}
	We swap the two sums on the right and we get
	\begin{equation*}
	\sum_{q=0}^N b_q= \sum_{i=1}^m \sum_{q=0}^N a_{q-h_i} = \sum_{i=1}^m c(N-h_i)^n+O(N^{n-1}).
	\end{equation*}
	On the other hand by \eqref{eq:iteratedHfunctions}, we have  $\sum_{q=0}^N b_q = c'N^n+O(N^{n-1})$.  It follows that $c'=mc$. Therefore, the differential symmetric signature of $R$ is $c/c'=1/m$ as required.
\end{proof}

\begin{Rem}
The proof of Theorem \ref{theorem:signatureofquotientsingularities} shows that the function $N\mapsto \sum_{q=0}^N\freerank_R\Sym_R^q(\Omega_{R/\kk})^{**}$ is quasi-polynomial. 
A similar behaviour for the F-signature function of quotient singularities is observed by the second author and De Stefani in \cite{CDS17}.
\end{Rem}

\section{Symmetric signature of hypersurfaces with isolated singularities}\label{section-hypersurfaces}
\par Let $\kk$ be an algebraically closed field, let $n\geq4$, and let $f\in\kk[x_1,\dots,x_n]$ be a homogeneous polynomial of degree $\geq2$ such that $\chara\kk\nmid\deg f$.
We consider the $(n-1)$-dimensional quotient ring $R=\kk[x_1,\dots,x_n]/(f)$, and assume in addition that $R$ is a domain with an isolated singularity.
The ring $R$ is the coordinate ring of a hypersurface in the $n$-dimensional affine space $\mathbb{A}_{\kk}^n$.
The goal of this section is to prove that $s_{\dd\sigma}(R)=0$. 
We will split the proof into several lemmas.
\par The arguments we employ here have been suggested by Bruns.

\begin{Prop}\label{prop:symisreflexive}
	Let $\kk$ be an algebraically closed field, let $n$ and $s$ be positive integers such that $n\geq2s+2$, and let $f_1,\dots,f_s\in\kk[x_1,\dots,x_n]$ be homogeneous polynomials of degrees $\geq2$ such that the ring $R=\kk[x_1,\dots,x_n]/(f_1,\dots,f_s)$ is a complete intersection domain with an isolated singularity. 
	Then the symmetric algebra 
	\begin{equation*}
	\mathcal{S}:=\displaystyle\bigoplus_{q\in\mathbb{N}}\Sym_R^q(\Omega_{R/\kk}).
	\end{equation*}
	is a complete intersection and $R$-reflexive.
\end{Prop}

\begin{proof}
    First, we prove that $\mathcal{S}$ is a complete intersection.
	We consider the corresponding sheaf of algebras over $X=\Spec R$, and its spectrum $T_{X/\kk}:=\Spec\widetilde{\mathcal{S}}$, i.e. the tangent scheme of $X$ in the sense of \cite[Section 16.5]{EGAIV}.
	Since $R$ is an isolated singularity, the tangent scheme $T_{X/\kk}$ has dimension $2d$ over the punctured spectrum of $R$, where $d=\dim R$.
	On the other hand the fiber over  $\mathfrak{m}_R=(x_1,\dots,x_n)  \in \Spec R $ is given by
	\begin{equation}\label{eq-fiberorigin}
	\displaystyle\bigoplus_{q\in\mathbb{N}}\Sym_R^q(\Omega_{R/\kk})\otimes_R\kk\cong\displaystyle\bigoplus_{q\in\mathbb{N}}\Sym_R^q(\Omega_{R/\kk}\otimes_R\kk).
	\end{equation}
	By \cite[Proposition 8.7]{Har77}, we have an isomorphism of $\kk$-vector spaces 
	$\Omega_{R/\kk}\otimes_R\kk \cong\mathfrak{m}_R/\mathfrak{m}_R^2$.  
	It follows that $\Omega_{R/\kk}\otimes_R\kk$ has $\kk$-dimension $n$, and therefore the fiber ring \eqref{eq-fiberorigin} has Krull dimension $n$, so $\dim T_{X/\kk}=\max\{2d,n\}$.
	Since $R$ is a complete intersection, we have $n=d+s$, therefore $\dim\mathcal{S}=\dim T_{X/\kk}=2d$, because $n\geq 2s$ implies $2d\geq n$.
	Finally, we have a natural identification $\mathcal{S}\cong\kk[x_1\dots,x_n,\dd x_1,\dots,\dd x_n]/(f_1,\dots,f_s,\dd f_1,\dots,\dd f_s)$, which shows that $\mathcal{S}$ is a complete intersection of dimension $2d=2n-2s$ in a polynomial ring with $2n$ variables.
	\par We prove that $\mathcal{S}$ is $R$-reflexive.
	Since $\mathcal{S}$ is a complete intersection, it is a Cohen-Macaulay ring, so it satisfies condition $(S_2)$.
	Moreover, the singular locus of $T_{X/\kk}$ is given by the fiber over the origin, so it has dimension $n$ and codimension $2d-n=n-2s\geq2$.
	In particular, $T_{X/\kk}$ is regular in codimension $1$.
	So, if $U$ is the punctured spectrum of $R$, and $\pi:T_{X/\kk}\rightarrow X$ the canonical projection, then $\pi^{-1}(U)\subseteq T_{X/\kk}$ is regular.
	Thus we have
	\begin{equation*}
	\Gamma(U,\widetilde{\mathcal{S}})\cong\Gamma(\pi^{-1}(U),\mathcal{O}_{T_{X/\kk}})=\mathcal{S},
	\end{equation*}
	therefore $\mathcal{S}$ is $R$-reflexive by Lemma \ref{lemma:reflexivehull}.
\end{proof}

\par Since direct summands of reflexive modules are reflexive, we obtain the following.

\begin{Cor}\label{cor:symqisreflexive}
Let $\kk$ and $R$ be as in Proposition \ref{prop:symisreflexive}. Then for every natural number $q$ the $R$-module $\Sym^q_R(\Omega_{R/\kk})$ is reflexive.	
\end{Cor}

\par We will use the following easy observation repeatedly.

\begin{Rem}\label{remark-unit} Let $R$ be a positively graded ring with homogeneous maximal ideal $\mathfrak{m}_R$ and let $I=(a_1,\dots,a_n)$ be an ideal of $I$. 
If $1\in I$ then at least one of $a_1,\dots,a_n$ is a unit. 
In fact, if none of  $a_i$'s is a unit, then $I\subseteq \mathfrak{m}_R$ by degree reasons. 
\end{Rem}

\begin{Lemma}\label{lemma:omegasummand}
 Let $P=\kk[x_1,\dots,x_n]$ be a polynomial ring over an algebraically closed field $k$. Consider $f_1,\dots,f_s\in P$ homogeneous elements of degrees $\geq2$ and the corresponding quotient ring $R=\kk[x_1,\dots,x_n]/(f_1,\dots,f_s)$.
 If $\freerank_R\Sym^q_R(\Omega_{R/\kk})>0$ for some $q\geq1$, then also $\freerank_R\Omega_{R/\kk}>0$.
\end{Lemma}

\begin{proof}
	 We define $\mathbb{N}^n_q:=\{\nu\in\mathbb{N}^n: \ \text{such that} \ \nu_1+\cdots+\nu_n=q\}$, and for every $\nu\in\mathbb{N}_q^n$  we use the notation $(\dd X)^{\nu}:=(\dd x_1)^{\nu_1}\cdots (\dd x_n)^{\nu_n}\in\Sym^q_R(\Omega_{R/\kk})$. 
	Moreover, we denote by $e_1,\dots,e_n$ the standard vectors of $\mathbb{N}^n$.
\par Since  $\freerank_R\Sym^q_R(\Omega_{R/\kk})>0$, there exists an $R$-module homomorphism 	 $\varphi:\Sym^q_R(\Omega_{R/\kk})\rightarrow R$, $(\dd X)^{\nu}\mapsto a_{\nu}\in R$ which contains $1$ in the image.
Thanks to Remark \ref{remark-unit}, we may assume without loss of generality that there exists $\mu\in\mathbb{N}^n_{q-1}$ such that $a_{\mu+e_1}=1$.
Applying $\varphi$ to the equations $(\dd X)^{\mu}\left(\frac{\partial f_i}{\partial x_1}\dd x_1+\cdots+\frac{\partial f_i}{\partial x_n}\dd x_n\right)=0$ in $\Sym^q_R(\Omega_{R/\kk})$ for $i=1,\dots,s$, we obtain $s$ relations in $R$ 
\begin{equation}\label{eq:Bruns1}
\frac{\partial f_i}{\partial x_1}+\frac{\partial f_i}{\partial x_2}a_{\mu+e_2}+\cdots+\frac{\partial f_i}{\partial x_n}a_{\mu+e_n}=0.
\end{equation}
Now, consider the map $\psi:P\rightarrow P$ given by $\psi(g)=\frac{\partial g}{\partial x_1}+\frac{\partial g}{\partial x_2}a_{\mu+e_2}+\cdots+\frac{\partial g}{\partial x_n}a_{\mu+e_n}$.
By \eqref{eq:Bruns1} we have $\psi(f_i)\in(f_1,\dots,f_s)$ for all $i=1,\dots,s$, therefore the map $\psi$ is a derivation on $R$, so it corresponds to an $R$-linear map $\Omega_{R/\kk}\rightarrow R$. Evaluating this map on $\dd x_1$, we see immediately that $1$ is in the image, therefore $\freerank_R\Omega_{R/\kk}>0$ as required.
\end{proof}

The following lemma gives a useful characterization of the vanishing of the free rank of $\Omega_{R/\kk}$.

\begin{Lemma}\label{lemma:charactOmegafreerank}
	Let $\kk$ be an algebraically closed field, let $f_1,\dots,f_s\in\kk[x_1,\dots,x_n]$ be homogeneous polynomials and consider the $\kk$-algebra $R=\kk[x_1,\dots,x_n]/(f_1,\dots,f_s)$. 
	Then $\freerank_R\Omega_{R/\kk}>0$ if and only if there is a column in the Jacobi matrix of $f_1,\dots,f_s$ which can be expressed as an $R$-linear combination of the other columns.
\end{Lemma}

\begin{proof}
	Let $M=\left(\frac{\partial f_i}{\partial x_j}\right)_{\substack{i=1,\dots,s\\j=1,\dots,n}}$ be the Jacobi matrix of $f_1,\dots,f_s$, and let $M^T$ be its transpose. From the short exact sequence
	\begin{equation*}
	R^s\xrightarrow{M^T}R^n \rightarrow \Omega_{R/\kk}\rightarrow 0
	\end{equation*}
	we see that a surjection $\Omega_{R/\kk}\twoheadrightarrow R$ is the same as a surjection $R^n\twoheadrightarrow R$ which vanishes on the image of $M^T$.
	To give such a surjection is equivalent to give an element $u\in R^n$  such that  $u\circ M^T = 0$ (where the composition is the usual matrix multiplication) and such that the entries of $u$ generate the unit ideal.
	Thus, this is equivalent to say that there is a linear dependence between the rows of $M^T$ (i.e. the columns of $M$) such that the coefficients generate the unit ideal in $R$.
	By Remark \ref{remark-unit}, it is the same as the existence of a column in the Jacobi matrix which can be expressed as a linear combination of the other columns.
\end{proof}

\begin{Lemma}\label{lemma:omeganosummand}
Let $\kk$ be an algebraically closed field and let $f\in\kk[x_1,\dots,x_n]$ be a homogeneous polynomial of degree $\geq2$ such that $\chara \kk\nmid \deg f$.
Assume that the quotient ring $R=\kk[x_1,\dots,x_n]/(f)$  is a domain with an isolated singularity. Then $\freerank_R\Omega_{R/\kk}=0$.
\end{Lemma}

\begin{proof}
Assume by contradiction that $\freerank_R\Omega_{R/\kk}>0$, then by Lemma \ref{lemma:charactOmegafreerank} there exist $a_2,\dots,a_n\in P=\kk[x_1,\dots,x_n]$ such that 
$\frac{\partial f}{\partial x_1}=\frac{\partial f}{\partial x_2}a_2+\cdots+\frac{\partial f}{\partial x_n}a_n$ in $R$.
 This can be lifted to a similar equation in the polynomial ring $P$, namely 
\begin{equation}\label{eq:Bruns2}
\frac{\partial f}{\partial x_1}=\frac{\partial f}{\partial x_2}a_2+\cdots+\frac{\partial f}{\partial x_n}a_n+bf,
\end{equation} 
 for some $b\in P$.
Since $\deg f$ is invertible in $\kk$, we can insert the Euler equation $f=\frac{1}{\deg f}\left(\frac{\partial f}{\partial x_1}x_1+\cdots+\frac{\partial f}{\partial x_n}x_n \right)$ into \eqref{eq:Bruns2}.
\par  We obtain that $\left(1-\frac{b}{\deg f}x_1\right)\frac{\partial f}{\partial x_1}\in\left(\frac{\partial f}{\partial x_2},\cdots,\frac{\partial f}{\partial x_n}\right)$ in $P$.
 This implies $\frac{\partial f}{\partial x_1}\in\left(\frac{\partial f}{\partial x_2},\cdots,\frac{\partial f}{\partial x_n}\right)$ in $P_{\mathfrak{m}}$, where $\mathfrak{m}=(x_1,\dots,x_n)$.
Therefore the Krull dimension of the singular locus $\left(f, \frac{\partial f}{\partial x_1},\cdots,\frac{\partial f}{\partial x_n}\right) = \left(\frac{\partial f}{\partial x_2},\cdots,\frac{\partial f}{\partial x_n}\right) $ is positive, and this contradicts the hypothesis of $R$ having an isolated singularity at the origin. 
\end{proof}

\begin{Theorem}\label{theorem:freerankiszero}
Let $\kk$ be an algebraically closed field, let $n\geq4$, and let $f\in\kk[x_1,\dots,x_n]$ be a homogeneous polynomial of degree $\geq2$ such that $\chara \kk\nmid \deg f$.
Assume that the quotient ring $R=\kk[x_1,\dots,x_n]/(f)$  is a domain with an isolated singularity.
Then the following facts hold.
\begin{enumerate}
\item For every $q\geq1$, $\freerank_R\Sym^q_R(\Omega_{R/\kk})=0$. 
\item The differential symmetric signature of $R$ is $s_{\dd\sigma}(R)=0$.
\end{enumerate}
\end{Theorem}

\begin{proof}
\par By Corollary \ref{cor:symqisreflexive}, the module $\Sym^q_R(\Omega_{R/\kk})$ is reflexive, so for the differential symmetric signature it is enough to consider the free rank and the rank of it, that is
\begin{equation*}
s_{\dd\sigma}(R)=\lim_{N\rightarrow+\infty}\frac{\sum_{q=0}^N\freerank_R\Sym^q_R(\Omega_{R/\kk})}{\sum_{q=0}^N\rank_R\Sym^q_R(\Omega_{R/\kk})}.
\end{equation*}
Therefore the first statement implies the second.
To prove \text{(1)}, we observe that by Lemma \ref{lemma:omeganosummand} $\freerank_R\Omega_{R/\kk}=0$, which implies $\freerank_R\Sym^q_R(\Omega_{R/\kk})=0$ by Lemma \ref{lemma:omegasummand}.
\end{proof}

\begin{Ex}
The conclusion of Theorem \ref{theorem:freerankiszero} is not true if $n\leq3$. 
For example, the two-dimensional hypersurface $R=\kk[x,y,z]/(x^2-yz)$ is a quotient singularity $R\cong\kk[u,v]^{\mathbb{Z}/2\mathbb{Z}}$, where the group is acting by negating each variable.
Therefore, by Theorem \ref{theorem:signatureofquotientsingularities} if $\chara\kk\neq2$, we have $s_{\dd\sigma}(R)=\frac{1}{2}\neq0$.
\end{Ex}

One can use Theorem \ref{theorem:freerankiszero} to exhibit an example of a ring where the differential symmetric signature and the F-signature are different.

\begin{Ex}\label{ex:differentsignatures}
Let $\kk$ be an algebraically closed field of characteristic $p>2$, and consider the hypersurface $R=\kk[x,y,z,w]/(xy-zw)$. 
The ring $R$ is the Segre product of two polynomial	rings in two variables, so by a result of Singh \cite[Example 7]{Sin05} the F-signature of $R$ is $s(R)=\frac{2}{3}$.
On the other hand, $R$ is an isolated singularity, so by Theorem \ref{theorem:freerankiszero} we have $s_{\dd\sigma}(R)=0$.
\end{Ex}

\begin{Ex}
Consider the polynomial ring $P=\kk[x_1,\dots,x_6]$ and the generic matrix
\begin{equation*}
X=\begin{pmatrix}
x_1 & x_2 & x_3\\
x_4 & x_5 & x_6
\end{pmatrix}.
\end{equation*}
Let $I=I_2(X)$ be the ideal generated by the $2\times2$ minors, and consider the quotient ring $R=P/I$.
It is well-known that $R$ is an isolated singularity of dimension $4$.
We show that $s_{\dd\sigma}(R)=0$.
A computation with Macaulay 2 \cite{M2} shows that the symmetric algebra $\bigoplus_{q\in\mathbb{N}}\Sym_R^q(\Omega_{R/\kk})$ is Cohen-Macaulay.
Thus, arguing as in the second part of the proof of Proposition \ref{prop:symisreflexive} we conclude that it is $R$-reflexive.
Therefore, by Lemma \ref{lemma:omegasummand} to prove that $s_{\dd\sigma}(R)=0$ is enough to show that $\freerank_R\Omega_{R/\kk}=0$.
Let $f_1=x_1x_5-x_2x_4$, $f_2=x_1x_6-x_3x_4$, $f_3= x_2x_6-x_3x_5$ be a system of generators for $I$. Their Jacobi matrix is 
\begin{equation*}
M= \begin{pmatrix}
x_5 & -x_4 & 0 & -x_2 & x_1 & 0 \\
x_6 & 0 & -x_4 & -x_3 & 0 & x_1 \\
0 & x_6 & -x_5 & 0 & -x_3 & x_2
\end{pmatrix}.
\end{equation*} 
Since no column of $M$ is an $R$-linear combination of the others, by Lemma \ref{lemma:charactOmegafreerank} we have $\freerank_R\Omega_{R/\kk}=0$ as desired.	
\end{Ex}

\par Partly motivated by the previous examples, the first author, Jeffries, and N\'u\~nez-Betancourt \cite{BJNB17} take another characteristic-free approach to the signature. They consider the \emph{principal part signature}, which is given by the asymptotic behavior of the ratio $\freerank_R P^n_{R/\kk}/\rank_RP^n_{R/\kk}$ of the modules $P^n_{R/\kk}$ of $n$-th principal parts of $R$. In this approach all toric rings have positive signature.

\section*{Acknowledgements}
We would like to thank Winfried Bruns for suggesting the ideas of Section \ref{section-hypersurfaces} and his interest in this project.
We thank Hai Long Dao for pointing out a mistake in an earlier version of this paper.

\end{document}